\documentclass[11pt]{amsart}
\usepackage{geometry}

\usepackage{graphicx} % Required for inserting images
\usepackage{amsmath}
\usepackage{amsfonts}
\usepackage{amssymb}
\usepackage{amsthm}

\newtheorem{theorem}{\rm\bf Theorem}[section]
\newtheorem{proposition}[theorem]{\rm\bf Proposition}
\newtheorem{lemma}[theorem]{\rm\bf Lemma}
\newtheorem{remark}[theorem]{\rm\bf Remark}
\newtheorem{problem}[theorem]{\rm\bf Problem}

\usepackage{tikz}

\newcommand{\Area}{\operatorname{Area}}
\newcommand{\sys}{\operatorname{sys}}
\newcommand{\inj}{\operatorname{inj}}
\newcommand{\conj}{\operatorname{conj}}
\usepackage{microtype}

\newtheorem{thm}{Theorem}[section]
\newtheorem{cor}[thm]{Corollary}

\newtheorem{lem}[thm]{Lemma}
\newtheorem{prop}[thm]{Proposition}

\title{Every closed surface of genus at least $18$ is Loewner}
\author{Qiongling Li and Weixu Su}
\address{Qiongling Li: Chern Institute of Mathematics and LPMC, Nankai University, Tianjin 300071, China} \email{qiongling.li@nankai.edu.cn}

\address{Weixu Su: School of Mathematics, Sun Yat-sen University, Guangzhou 510275, China}
\email{suwx9@mail.sysu.edu.cn}

\date{\today}

\begin{document}

\begin{abstract} 
In this paper, we obtain an improved  upper bound involving the systole and area for the volume entropy of a Riemannian surface. As a result, we show that every orientable and closed Riemannian surface of genus $g\geq 18$ satisfies Loewner's systolic ratio inequality. We also show that every closed orientable and nonpositively curved Riemannnian surface of genus $g\geq 11$  satisfies Loewner's systolic ratio inequality.

 \medskip

\noindent {\bf Keywords:} Systole; Loewner's inequality; entropy. 

\medskip

\noindent  {\bf MSC2020:} {53C23, 37C35.}

	\end{abstract}

\thanks{Q. Li is partially supported by the National Key R$\&$D Program of China No. 2022YFA1006600, the Fundamental Research
Funds for the Central Universities and Nankai Zhide foundation. W. Su is partially supported by NSFC No. 12371076.}

\maketitle

\section{Introduction}

Unless otherwise stated, all surfaces we consider in this paper are assumed to be orientable and closed, endowed with Riemannian metrics.

Let $M$ be a surface of genus $g\geq 1$. Denote the \emph{systole} of $M$ by $\sys(M)$. By definition, $\sys(M)$ is the length of shortest non-contractible closed geodesics on $M$.  Denote the area of $M$ by 
$\Area(M)$. We say that the surface $M$ is \emph{Loewner} if the systolic ratio satisfies the inequality
\begin{equation}\label{equ:Loewner}
    \frac{\sys^2(M)}{\Area(M)} \leq \frac{2}{\sqrt{3}}.
\end{equation}

Around 1949, 
Loewner (ref. \cite{Pu}) first proved that every surface of genus $1$ satisfies the inequality \eqref{equ:Loewner}. 
Gromov \cite{gromov1983} showed that, for any surface of genus $g$, 
\begin{equation*}
    \frac{\sys^2(M)}{\Area(M)}< \frac{64}{4\sqrt{g}+27}.
\end{equation*}
Thus all surfaces of genus $g\geq 50$ is Loewner. 
For details on history of Loewner's inequality and related developments in  the systolic geometry,
we refer to the book of Katz \cite{katz2007}.
See also Guth \cite{Guth}.

\begin{problem}
   Whether or not all surfaces of genus at least $1$
are Loewner?
\end{problem}

For surfaces of genus $2$ and genus $g\geq 20$, the answer is affirmative, shown by Katz and Sabourau \cite{ks2005,ks2006}. 

This aim of this short note is to prove

\begin{theorem}\label{thm:loewner}
Every closed Riemannian surface of genus $g\geq 18$ is Loewner.
\end{theorem}

Our proof of Theorem \ref{thm:loewner} uses the argument in \cite{ks2005}. We improve the upper bound for the volume entropy, which gives a better bound on the genus when the surface is not Loewner. 

For $3\leq g\leq 17$, the problem remains unsolved. In Proposition \ref{prop:small}, we show that all surfaces of 
of genus at least $17$ satisfying  
\begin{equation}\label{equ:injectivesystol}
\inj(M)=\sys(M)/2
\end{equation}
are Loewner. The hypothesis \eqref{equ:injectivesystol} is true for nonpositively curved surfaces, see
Proposition \ref{prop:injsys}.
We can do even better for nonpositively curved surfaces using the method of Gromov and Bishop-G\"unther volume comparison theorem. 
\begin{theorem}\label{thm:nonpositive}
    Every closed nonpositively curved Riemannnian surface of genus $g\geq 11$ is Loewner. 
\end{theorem}
\bigskip

\textbf{Acknowledgements.} 
We are grateful to Christpher Croke, Mikhail Katz and St\'ephane Sabourau  for their very useful comments.

\section{An upper bound for the entropy}

\subsection{Entropy}

Let $M$ be a surface of genus $g$. 
Let $\widetilde M$ be the universal cover of $M$. 
The \emph{volume entropy} of $M$ is defined by 
$$h(M)=\lim_{R\to +\infty} \frac{\log \operatorname{Area}(B(\tilde{x}_0, R))}{R}.$$
The limit exists and does not depend on the choice of $\tilde{x}_0\in \widetilde{M}$ (see \cite{Manning}). 

 By a classical result of Katok \cite{katok1983}, we have 
$$h(M)^2 \geq \frac{2\pi (2g-2)}{\Area(M)}.$$

The following equivalent definition of $h(M)$ is very useful.

\begin{lemma}\label{lem:entropy}
    Let $N(T)$ be the number of homotopy classes of non-contractible closed curves on $M$, based at some fixed point $x_0\in M$ and which can be represented by loops of length at most $T$. Then
    $$h(M)= \lim_{T\to +\infty} \frac{\log N(T)}{T}.$$
\end{lemma}
See \cite[Lemma 3.6]{ks2005} for the proof. 

\subsection{Upper bound}

The next inequality is the key and technical part of this paper. It relates the volume entropy to the systolic ratio. It improves \cite[Proposition 3.1]{ks2005}.

\begin{proposition}\label{thm:entropy}
    Every surface $M$ of genus $g\geq 1$ satisfies 
    $$h(M) \leq \frac{1}{\beta\sys(M)} \log \frac{\Area(M)-\min\limits_{q\in M} \Area (B(q,(\beta-3\alpha)\sys(M)))}{\min\limits_{q\in M} \Area (B(q,\alpha\sys(M)))}, $$
    whenever $0<5\alpha<\beta, \beta+4\alpha< \frac{1}{2}$. 
\end{proposition}

\begin{proof}

Let $s$ be the systole of $M$.
Let $\mathcal S$ be a maximal set of centers of disjoint disks of radius $\alpha s$. So $\{B(q,2\alpha s) \ | \  q\in \mathcal S\}$ forms a cover of $M$.

For each closed geodesic path $c$ based at $x_0$ of length $\leq T$, we can always construct a homotopic path connected by at most $m:=\lceil\frac{T}{\beta s}\rceil$ points in the set $\mathcal S$ in the following way.

Let $t\in [2,3)$ to be determined. 
We will choose $\underline c(k)=c(\gamma_k s)\in c$ and $q_k\in \mathcal S$ in each step.
We start with $\underline c(0)=c(0)=x_0$ (thus $\gamma_0=0$). Choose $q_0\in \mathcal S$ such that $c(0)\in B(q_0, 2\alpha s)$.

We define by induction. Suppose we have defined $\gamma_k$ and $q_k$. By the construction, $\underline c(k)=c(\gamma_k s)\in B(q_k, 2\alpha s)$. Now we are going to define $\gamma_{k+1}$ and $q_{k+1}$.

We separate into two cases.
Note that in the following discussion,
the constants $\alpha, \beta$ satisfy
$$\beta>5\alpha, \beta+4\alpha<\frac{1}{2}.$$

\textbf{Case 1}: Suppose that any $q\in \mathcal S$ such that $d(c((\gamma_k+\beta+t\alpha)s), q)<2\alpha s$ satisfies $$d(q_k,q)\geq (\beta-t\alpha)s.$$

Choose $\hat q_1\in \mathcal S$ such that $$d(c((\gamma_k+\beta)s),\hat q_1)<2\alpha s.$$ Let $r_1\in [(\gamma_k+\beta)s,(\gamma_k+\beta+t\alpha)s]$ be the minimal number such that $d(\underline c(r),\hat q_1)<2\alpha s$ does not hold. So $d(c(r_1),\hat q_1)= 2\alpha s$. Then choose  $\hat q_2\in \mathcal S$ such that $d(c(r_1),\hat q_2)<2\alpha s.$  We continue the process to choose $r_k$'s and $\hat q_k$'s. Note that $$d(\hat q_l,\hat q_{l+1})\leq d(c(r_l),\hat q_l)+d(c(r_l),\hat q_{l+1})<4\alpha s.$$

\begin{figure}[htb]
	\begin{tikzpicture}
		\draw [thick] (-5,0) -- (7,0);
  \draw[fill] (-3.8,0) circle [radius=0.05];
  \node [below] at (-3.8,0) {$\underline c(k)$};
  \draw[fill] (3,0) circle [radius=0.05];
  \node [below] at (2.5,-0.3) {$c((\gamma_k+\beta)s)$};
   \draw[fill] (6,0) circle [radius=0.05];
  \node [below] at (7,0) {$c((\gamma_k+\beta+t\alpha)s)$};
\draw [red, ultra thick] (-3.5,-0.8) circle [radius=1];
  \draw[fill] (-3.5,-0.8) circle [radius=0.05];
  \node [below] at (-3.5,-0.8) {$q_k$};
  \draw [red, ultra thick] (3,0.7) circle [radius=1];
  \draw[fill] (3,0.7) circle [radius=0.05];
  \node [above] at (3,0.7) {$\hat q_1$};
  \draw [red, ultra thick] (4.5,-0.2) circle [radius=1];
  \draw[fill] (4.5,-0.2) circle [radius=0.05];
  \node [below] at (4.5,-0.2) {$\hat q_2$};
	\end{tikzpicture}
	\caption{Picture in Case 1. The disks $B(\hat{q}_k)$'s cover the segment connecting $c((\gamma_k+\beta)s)$ 
 and $c((\gamma_k+\beta+t\alpha)s)$. There is a smallest $l_0$  such that $d(q_k,\hat q_{l_0})\geq (\beta-t\alpha)s.$} 
	\label{fig_04}
\end{figure}
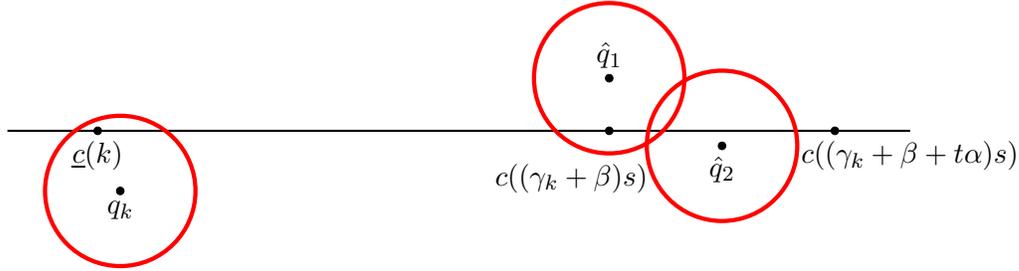

Let $l_0$ be the smallest integer such that $d(q_k,\hat q_{l_0})\geq (\beta-t\alpha)s.$
In fact, the distance between $c((\gamma_k+\beta+t\alpha)s)$
and $q_k$ is not less than $\beta+t\alpha-2\alpha$, and
if $c((\gamma_k+\beta+t\alpha)s)$ is contained in some disk 
$B(\hat q_j, 2\alpha)$ where $\hat q_j\in \mathcal S$, then 
$d(q_k, \hat q_j)> (\beta+t\alpha-4\alpha)s \geq (\beta-t\alpha)s$. Thus such an $l_0$ exists. 

Set $q_{k+1}=\hat q_{l_0}$ and $\underline c(k+1)=c(r_{l_0}).$

If $l_0=1$,then $\underline c(k+1)=c((\gamma_k+\beta)s).$ The length of the circle joining $\underline c(k), q_k, \underline c(k+1), q_{k+1}$ is less than 
\[2\alpha s+(\beta+4\alpha)s+2\alpha s+\beta s=(2\beta+8\alpha)s<s.\]
So the circle is null-homotopic.

If $l_0>1$,then $d(q_k,\hat q_{l_0-1})< (\beta-t\alpha)s.$  
So \[d(q_k,\hat q_{l_0})\leq d(q_k,\hat q_{l_0-1})+d(\hat q_{l_0-1},\hat q_{l_0})\leq (\beta-t\alpha)s+4\alpha s=(\beta-t\alpha+4\alpha)s.\]
The length of the circle joined by $\underline c(k), q_k, \underline c(k+1), q_{k+1}$ is less than 
\[2\alpha s+(\beta-t\alpha+4\alpha)s+2\alpha s+(\beta+t\alpha)s=(2\beta+8\alpha)s<s.\]
Thus the circle is also null-homotopic.\\

\textbf{Case 2:} Suppose there exists $q\in \mathcal S$ such that $d(c((\gamma_k+\beta+t)s), q)<2\alpha s$ and $d(q_k,q)<(\beta-t\alpha)s.$ Then set $\underline c(k+1)=c((\gamma_k+\beta+t\alpha)s)$ and $q_{k+1}=q.$ The length of the circle joined by $\underline c(k), q_k, \underline c(k+1), q_{k+1}$ is less than 
\[2\alpha s+(\beta-t\alpha)s+2\alpha s+(\beta+t\alpha)s=(2\beta+4\alpha)s<s.\]
So the circle is null-homotopic.\\

Let $\mathcal R$ denote the set of all closed geodesics based at $x_0$ with length less than $T$. Now the above process allows us to construct for each closed geodesic $c\in \mathcal R$ a homotopic piecewise geodesic loop  by connecting points in $\mathcal S$.
 From the construction, each such representation of $c$ satisfies:
\begin{itemize}
    \item Either $d(q_k, q_{k+1})\geq (\beta-t\alpha)s$ and $\beta s\leq d(\underline c(k), \underline c(k+1))\leq (\beta+t\alpha)s,$ or,
    \item  $d(q_k, q_{k+1})< (\beta-t\alpha)s$ and $d(\underline c(k), \underline c(k+1))\geq (\beta+t\alpha) s.$ 
\end{itemize}

In the latter case, we have
\begin{eqnarray*}
    d(q_k, q_{k+1}) &\geq& d(\underline c(k), \underline c(k+1))- 4\alpha s \\
    &\geq& (\beta+t\alpha-4\alpha) s.
\end{eqnarray*}

If we take $t=2$, then  in both cases we have
\begin{equation}\label{equ:gap}
    d(q_k, q_{k+1}) \geq (\beta-2\alpha)s. 
\end{equation}

Since $d(q_k, q_{k+1})\geq (\beta-2\alpha)s,$ we have $B(q_{k+1}, \alpha s)\in M\setminus B(q_k,(\beta-3\alpha)s).$ 
 Let $S^{(k)}$ be the set of $q_l$ satisfying $B(q_l, \alpha s)\in M\setminus B(q_k,(\beta-3\alpha)s).$ So after $q_k$ is determined, the choices of $q_{k+1}$ is at most 
\begin{eqnarray*}
  |S^{(k)}|&\leq& \frac{\Area (M)-\Area (B(q_k,(\beta-3\alpha)s))}{\min\limits_{q\in \mathcal{S}} \Area (B(q,\alpha s))}\\
  &\leq& \frac{\Area (M)- \min\limits_{p\in M}\Area (B(q,(\beta-3\alpha)s))}{\min\limits_{q\in M} \Area (B(q,\alpha s))}. \\
\end{eqnarray*}
Denote the last term by $I.$ 

Let $m=\lceil\frac{T}{\beta s}\rceil.$
Then 
\begin{eqnarray*}
|\mathcal R| &\leq & I^m.
\end{eqnarray*}

Therefore, by Lemma \ref{lem:entropy},
\[ h(M)\leq \lim_{m\rightarrow \infty}\frac{1}{m\beta s}\log |\mathcal R|\leq \frac{1}{\beta s}\log I.\]
\end{proof}

We say that a metric disk $B(x,R)$ on $M$ is \emph{small} if it is diffeomorphic to an two dimensional Euclidean disk and  all geodesics in $B(x,R)$ connecting two boundary points of $B(x,R)$ minimize the distance between the endpoints. 

Denote by $\inj(M)$ the injectivity radius of $M$.  Note that all metric disks of radius less than $\inj(M)/2$ are small.

We recall the following result of Croke \cite[Corollary 1.2]{croke2009}.

\begin{lemma}\label{equ:Croke}
If $B(x,R)$ is small metric disk of $M$, e.g. $R\leq \inj(M)/2$, then 
$$\Area (B(x,R)) \geq \frac{8-\pi}{2} R^2.$$
\end{lemma}

The following corollary gives the upper bound of volume entropy involving the area and the injectivity radius. It improves \cite[Proposition 3.6]{katok1983}.
\begin{cor}\label{cor:entropyinj}
    Every surface $M$ of genus $g\geq 1$ satisfies 
    $$h(M) \leq \frac{1}{(1-4\eta)\inj(M)} \log \frac{\frac{2\Area(M)}{\inj^2(M)}-(8-\pi)(\min\limits\{1-7\eta,\frac{1}{2}\})^2}{(8-\pi)\eta^2}, $$
    whenever $0<\eta<\frac{1}{9}$.
\end{cor}
\begin{proof}
Note that $\sys(M)\geq 2 \inj(M)$.

 From Proposition \ref{lem:entropy}, every surface $M$ of genus $g\geq 1$ satisfies 
    $$h(M) \leq \frac{1}{(1/2-4\alpha)\sys(M)} \log \frac{\Area(M)-\min\limits_{q\in M} \Area (B(q,(1/2-7\alpha)\sys(M)))}{\min\limits_{q\in M} \Area (B(q,\alpha\sys(M)))}, $$
    whenever $0<\alpha<1/18$.

Let $\eta$ satisfying $\alpha \sys(M)= \eta\inj(M)$. It follows from $\sys(M)\geq 2\inj(M)$ that  $\eta \geq 2\alpha$.
Assume that $0<\eta<1/9,$ then $0<\alpha<1/18.$ 
It also follows from $\sys(M)\geq 2\inj(M)$ that $$(1/2-7\alpha)\sys(M)\geq(1-7\eta)\inj(M).$$

By Lemma \ref{equ:Croke}, we have 
\[\Area(B(q,\alpha\sys(M))\geq \frac{8-\pi}{2}\eta^2\inj(M)^2,\]
\[\Area(B(q,(1/2-7\alpha)\sys(M)))\geq \frac{8-\pi}{2}(\min\{1-7\eta,\frac{1}{2}\})^2\inj(M)^2. \]

Note that $(1/2-4\alpha)\sys(M)\geq (1-4\eta)\inj(M).$ The statement follows. 
\end{proof}

We then have the following result.
\begin{prop}\label{prop:small}
Suppose $M$ is a surface of genus at least $17$ that satisfies $\inj(M)=\frac{1}{2}\sys(M)$. Then $M$ is Loewner. 
\end{prop}
\begin{proof}
By Corollary \ref{cor:entropyinj} and the assumption $\inj(M)=\frac{1}{2}\sys(M)$, we obtain 
\begin{eqnarray*}
    h(M)&\leq& \frac{1}{(1-4\eta)\inj(M)} \log \frac{\frac{2\Area(M)}{\inj^2(M)}-(8-\pi)(\min\limits\{1-7\eta,\frac{1}{2}\})^2}{(8-\pi)\eta^2} \\
    &=& \frac{1}{(1/2-2\eta)\sys(M)} \log \frac{\frac{8\Area(M)}{\sys^2(M)}-(8-\pi)(\min\limits\{1-7\eta,\frac{1}{2}\})^2}{(8-\pi)\eta^2} 
\end{eqnarray*}    whenever $0<\eta<\frac{1}{9}.$
Denote the last term by $\frac{1}{(1/2-2\eta)\sys(M)} \log \tilde{I}$. 

Using Katok's inequality of entropy, we have 

\[ \sqrt{\frac{2\pi(2g-2)}{\Area (M)}}\leq h(M) \leq \frac{1}{\beta \sys(M)} \log \tilde{I}. \]

Assume that $M$ is not Loewner, then $\frac{\Area(M)}{\sys(M)^2}<\frac{\sqrt{3}}{2}.$ Therefore,
\begin{eqnarray*}
    g-1
    &\leq& \frac{\sqrt{3}}{8\pi}\cdot \frac{(\log\frac{4\sqrt{3}-(8-\pi)(\min\limits\{1-7\eta,\frac{1}{2}\})^2}{(8-\pi)\eta^2})^2}{(1/2-2\eta)^2}\\
    &\approx& 15.9493,
\end{eqnarray*} where $\eta=0.065734.$
Hence $g\leq 16.$
\end{proof}

\section{Proof of Theorem \ref{thm:loewner}}
Let $(M,g)$ be a surface of genus $g\geq 1$.
Gromov \cite[Proposition 5.1.B]{gromov1983} showed that, 
there is a non-negative function $\rho(x)$ defined on $M$ such that 
every metric disk $B(x,R)$  with radius  $\rho(x)/2 \leq R\leq \sys(M)/2$ satisfies 
\begin{equation}\label{equ:gromov}
    \Area (B(x,R)) \geq \frac{1}{2} (2R-\rho(x))^2.
\end{equation}
The function $\rho(x)$ is called the \emph{height function} of $M$.  It is less than the diameter of $M$ for all surfaces of genus at least $1$.  See \cite[\S 5.1]{gromov1983}
    for details. 

Given any surface $M$ of genus $g\geq 1$, 
one can make conformal deformation of $M$ such that 
the height function becomes arbitrary small. Meanwhile,
the systole is preserved and the area is non-increasing.

\begin{lem}\label{lemma:smallheight}
For an arbitrary $\epsilon>0,$ there exists a Riemannian metric $\overline M$ conformal to $M$ such that 
\begin{itemize}
    \item $\sys(\overline M)=\sys(M).$
    \item $\Area(\overline M)\leq \Area(M).$
    \item The height function of $\overline M$ is less than $\epsilon.$
\end{itemize}
\end{lem}
This is proved by Gromov \cite[\S 5.6.C'']{gromov1983}. Note that 
$$\frac{\sys(\overline M)}{\Area(\overline M)} \geq \frac{\sys( M)}{\Area(M)}.$$

Now we can proceed the proof of Theorem \ref{thm:loewner}.

\begin{proof}
Assume  $M$ is not Loewner, then $$\frac{\sys(M)^2}{\Area(M)}>\frac{2}{\sqrt{3}}.$$ 

Let $\delta>0$ be arbitrarily small. By Lemma \ref{lemma:smallheight}, up to a conformal deformation,
we may assume that 
$\frac{\sys(M)^2}{\Area(M)}>\frac{2}{\sqrt{3}}$ and the height function of
$M$ is less than $\delta\sys(M)$. 
Thus for all $q\in M$ and  $\delta\sys(M)/2 \leq R\leq \sys(M)/2$,
$$\Area(B(q, R))\geq \frac{1}{2}(2R-\delta\sys(M))^2.$$

 Assume that $\delta/2\leq \alpha, \beta-3\alpha\leq 1/2$, then
$$\Area(B(q, \alpha\cdot \sys(M)))\geq \frac{1}{2}(2\alpha-\delta)^2\sys(M)^2,$$
$$\Area(B(q, (\beta-3\alpha)\cdot \sys(M)))\geq \frac{1}{2}(2\beta-6\alpha-\delta)^2\sys(M)^2.$$

By Proposition \ref{thm:entropy}, we have 

\begin{eqnarray*}
    h(M) &\leq & \frac{1}{\beta \sys(M)} \log \frac{\frac{\Area(M)}{\sys(M)^2}-\frac{1}{2}(2\beta-6\alpha-\delta)^2}{\frac{1}{2}(2\alpha-\delta)^2} \\
    &\leq& \frac{1}{\beta \sys(M)} \log \frac{\sqrt{3}-(2\beta-6\alpha-\delta)^2}{(2\alpha-\delta)^2},
\end{eqnarray*}
    whenever $2\delta<\alpha<\frac{1}{5}\beta, \beta+4\alpha< \frac{1}{2}$. 
Denote the last term by $\frac{1}{\beta \sys(M)} \log \tilde{I}$. 

Using Katok's inequality of entropy, we have 

\[ \sqrt{\frac{2\pi(2g-2)}{\Area (M)}}\leq h(M) \leq \frac{1}{\beta \sys(M)} \log \tilde{I}. \]

Let $\delta=0.000001,\alpha=0.026377, \beta=0.394491<0.394492=\frac{1}{2}-4\alpha.$ 
Using $\frac{\Area (M)}{\sys(M)^2}<\frac{\sqrt{3}}{2}$ again, we obtain 
\begin{eqnarray*}
    g-1 &<& \frac{\sqrt{3}}{8\pi}\frac{(\log \tilde{I})^2}{\beta^2} \\
    &=&\frac{\sqrt{3}}{8\pi}\frac{(\log \frac{\sqrt{3}-(2\beta-6\alpha-\delta)^2}{(2\alpha-\delta)^2})^2}{\beta^2} \\
    &\approx& 16.8728.
\end{eqnarray*}

Therefore, if $M$ is not Loewner, then $g\leq 17.$

\end{proof}

\section{Proof of Theorem \ref{thm:nonpositive}}
Assume that $M$ is nonpositively curved. By the Bishop-G\"unther volume comparison theorem (ref \cite[Theorem III.4.2]{chavelbook}), we have for all $q\in M$ and $0<R\leq \inj(M),$
    \[\Area(B(q,R))\geq \pi R^2.\]
    Note that $\pi R^2$ is the area of the disk of radius $R$ in the Euclidean plane.

\begin{lemma}\label{lem:CE}
   Let $M$ be a compact Riemannian manifolds. Then 
    the injectivity radius of $M$ equals the minimium of the
conjugate radius and half the length of a shortest nontrivial closed geodesic on
$M$. 
\end{lemma}
\begin{proof}
    This is a corollary of Cheeger-Ebin \cite[Lemma 5.6]{CE}. We give a proof for the sake of completeness.

     For $p\in M$,
denote the injectivity radius and conjugate radius 
at $p$ by $\inj_p(M)$ and $\operatorname{conj}_p(M)$,
respectively. 
And denote the shortest length of closed nontrivial geodesic on
$M$ by $\ell(M)$. It is obvious that $\inj_p(M)\leq \conj_p(M)$ and thus $\inj(M)\leq \conj(M).$

Suppose $\inj(M)< \operatorname{conj}(M)$. Assume that at some $p\in M$,  $\inj_p(M)=\inj(M)$. 
By Cheeger-Ebin \cite[Lemma 5.6]{CE}, there is a geodesic loop starting and ending at $p$, with length $2\inj_p(M)=2\inj(M)$. Let $q \in \gamma$ that separates $\gamma$ into two segments with the same length $\inj(M)$. Note that $p$ and $q$ are not conjugate since $\inj_p(M)<\operatorname{conj} (M)$. The existence of two geodesic segments between $p$ and $q$ implies that $\inj_q(M)=\inj(M)$. If we look at $\gamma$ as a loop starting and ending at
$q$, then $\gamma$ is also smooth at $p$ by using Cheeger-Ebin \cite[Lemma 5.6]{CE} again. Thus $\gamma$ is a smooth nontrivial closed geodesic.
 This shows that $\inj(M) \geq \frac{\ell(M)}{2}$.

Therefore,
$$\inj(M)=\min\{ \operatorname{conj}(M), \ell(M)/2   \}.$$

\end{proof}

\begin{proposition}\label{prop:injsys}
   Under the assumption that $M$ is nonpositively curved, we have 
   \begin{equation}\label{equ:injsys}
       \inj(M)=\frac{\sys(M)}{2}.
   \end{equation}
\end{proposition}
\begin{proof}

Since $M$ is nonpositively curved, there is no conjugate points on $M$.  
Moreover, any (nontrivial) closed geodesic $\gamma$ on $M$ is not contractible. Otherwise, 
$\gamma$ lifts to a closed geodesic on the universal cover $\widetilde{M}$. This contradicts the fact that  on a Cartan-Hadamard manifold, any two points are connected by a unique geodesic. 

Now \eqref{equ:injsys} is an immediate corollary of Lemma \ref{lem:CE}.

\end{proof}

Note that for a nonpositively curved surface $M$, we already know that the area of $M$ is at least $\pi\inj(M)^2=\frac{\pi}{4}\sys(M)^2.$  Thus $\frac{\Area(M)}{\sys(M)^2}\geq \frac{\pi}{4},$ while $\frac{\pi}{4}$ is strictly less than $\frac{\sqrt{3}}{2}$. One cannot conclude the Loewner property directly from the nonpositively curved condition. This is different from the case of hyperbolic metric, where we have a better lower bound of the area of $M$ to ensure it is Loewner. 
\bigskip

Now we can proceed the proof of Theorem \ref{thm:nonpositive}. 
\begin{proof}
We will use the argument of Gromov
\cite[Section 5.3]{gromov1983}.

Fix an arbitrary point $x\in M$. Choose some 
$p_0\in M\setminus B(x,\sys(M)/2)$ and choose  $\gamma_1,\cdots, \gamma_{2g}$ in $M\setminus B(x,\sys(M)/2)$ as a canonical basis of of $\pi_1(M,p_0)$. Denote $$\Gamma=\bigcup_{i=1}^{2g}\gamma_i.$$ 
We consider $\Gamma$ as a connected graph. 
The Betti number of a connected graph is given by
$1+E-V$ where $E$ is the number of edges and $V$ is the number of vertices. Thus the Betti number of $\Gamma$,
denoted by $b_1(\Gamma)$, is equal to $2g$.

Let $\mathcal S$ be a maximal set of centers on $\Gamma$ of disjoint disks of radius $\sys(M)/8$ such that $$\{B(q,\sys(M)/4) \ | \  q\in \mathcal S\}$$ forms a cover of $\Gamma$. We assume that $p_0\in \mathcal{S}$.

We first construct a covering of each $\gamma_i$ as following. Denote $\gamma=\gamma_i$ for simplicity.

Start with $p_0=\gamma(0)$. Set $q_0=p_0$. 
Let $t_1$ be the smallest positive number such that $d(\gamma(t_1), p_0=q_0)<\sys(M)/4$ does not hold. Define $p_1=\gamma (t_1)$ and thus $d(p_1,q_0)=\sys(M)/4$. Denote by 
$\overline{p_0p_1}$ the geodesic connecting $p_0$
to $p_1$ and so on. Since 
$\gamma([0, t_1]) \cup \overline{p_0p_1}$ is contained in 
$B(p_0, \sys(M)/4)$, 
the path $\gamma([0,t_1])$ is homotopic to $\overline{p_0 p_1}$.

Choose $q_1\in \mathcal{S}$ such that $d(p_1,q_1)<\sys(M)/4.$ Let $t_2$ be the smallest positive number such that $d(\gamma(t_2), q_1)<\sys(M)/4$ does not hold. Define $p_2=\gamma (t_2)$ and thus $d(p_2,q_1)=\sys(M)/4$. Then the path $\gamma([t_1,t_2])$ is homotopic to $\overline{p_1 p_2}.$

Repeat this process, we obtain a sequence of points $p_0=\gamma(0),\cdots, p_r=\gamma(t_r), p_{r+1}=p_0=\gamma(l)$ on $\gamma$ and $q_0,q_1,\cdots, q_r, q_{r+1}=q_0\in \mathcal S$ satisfying:
\[d(p_i,q_i)<\sys(M)/4,\quad d(p_{i+1},q_i)=\sys(M)/4.\]
And for $t\in [t_r,l]$, $d(\gamma(t),p_r)<\sys(M)/4.$ So the path $\gamma([t_r,l])$ is homotopic to $\overline{p_r p_0}.$

Then the triangle formed by $p_k,q_k, p_{k+1}$ has length less than $2(\sys(M)/4+\sys(M)/4)=\sys(M)$ and thus is null-homotopic. Also, the triangle formed by $p_{k+1}, q_k,q_{k+1}$ has length less than $2(\sys(M)/4+\sys(M)/4)=\sys(M)$ and thus is null-homotopic. Together, we obtain that the circle $\overline{p_kq_k}\cup\overline{q_kq_{k+1}}\cup\overline{p_{k+1}q_{k+1}}\cup\overline{p_kp_{k+1}}$
is homotopically trivial. Thus the curve $\gamma$ can be homotopic to a piecewise geodesic connected by points 
$q_0,q_1,\cdots, q_r, q_{r+1}=q_0$.

For each $\gamma_i$, we obtain that it is homotopic to a piecewise geodesic connected by points $p_0^{(i)},p_1^{(i)},\cdots,p_{r+1}^{(i)}=p_0^{(i)}$ lying in $\gamma_i$ and also a piecewise geodesic connected by points $q_0^{(i)},q_1^{(i)},\cdots,q_{r+1}^{(i)}=q_0^{(i)}$ lying in $\mathcal S.$ Note that $p_0^{(i)}=q_0^{(i)}=q_0.$

Therefore, we obtain an abstract connected graph $\Gamma'$ with vertices $q_k^{(i)}\in\mathcal S$ appeared above, $q, q'$ has an edge if and only if they appear as $q=q_k^{(i)}, q'=q_{k+1}^{(i)}$ for some $i$. There is a (not necessarily injective) map from $\Gamma$ to $\Gamma'$, taking $p_k^{(i)}$ to $q_k^{(i)}$ and the edge $\overline{p_k^{(i)}p_{k+1}^{(i)}}$ to $\overline{q_k^{(i)}q_{k+1}^{(i)}}.$ Denote the image of $\gamma_i$ in $\Gamma'$ by $\sigma_i.$ So $\sigma_i, i=1,\cdots,2g,$ are different nontrivial loops based at $q_0$. 

Moreover, there is a natural map $\iota:\Gamma'\rightarrow M$ by the natural map taking the edge between $q_k,q_{k+1}$ to a fixed geodesic between $q_k,q_{k+1}$. 
Note that by our construction, for $i=1,\cdots,2g$, $\gamma_i$ is homotopic to $\iota(\sigma_i).$

Suppose there is a loop $\sigma_{i_1}^{n_1}\cdots\sigma_{i_l}^{n_l}$ is homotopically trivial in $\Gamma'$. It is clear that $\iota(\sigma_{i_1}^{n_1}\cdots\sigma_{i_l}^{n_l})\in \pi_1(M)$ is also homotopically trivial. Since $\gamma_i$ is homotopic to $\iota(\sigma_i)$, we obtain $\gamma_{i_1}^{n_1}\cdots\gamma_{i_l}^{n_l}$ is also homotopically trivial. 

Therefore, $\pi_1(\Gamma')$ has at least $2g$ generators and thus the Betti number of $\Gamma'$, $b_1(\Gamma')$, is not less than  $2g$. 

Let $N$ denote the number of vertices in $\Gamma'$. So $N\leq |\mathcal S|$. Denote the number of edges in $\Gamma'$ by $K$. Then $K\leq \frac{N(N-1)}{2}$. So
\begin{eqnarray*}
    2g\leq b_1(\Gamma')
    & \leq&\frac{N(N-1)}{2}-N+1 \\
&=&\frac{(N-1)(N-2)}{2} \\
& \leq& \frac{(|\mathcal S|-1)(|\mathcal S|-2)}{2} .
\end{eqnarray*}
Note that 
\[\Area(B(p,\sys(M)/8)) \geq \pi \left(\sys(M)/8\right)^2,\quad\Area(B(p,3\sys(M)/8)) \geq \pi \left(3\sys(M)/8\right)^2.\]
Since the disjoint balls $B(q,\sys(M)/8)(q\in \mathcal S)$ lie in $M\setminus B(x, 3\sys(M)/8)$, we have 
\begin{eqnarray*}
     |\mathcal{S}| &\leq& \frac{\Area(M)-\Area(B(x,\frac{3}{8}\sys(M)))}{\pi (\sys(M)/8)^2}\\
    &\leq& \frac{\Area(M)-\pi(\frac{3}{8}\sys(M))^2}{\pi (\sys(M)/8)^2}\\
    &\leq& \frac{64}{\pi}\frac{\Area(M)}{\sys(M)^2}-9.
\end{eqnarray*}

Assume that $M$ is not Loewner, then $\frac{\Area(M)}{\sys(M)^2}<\frac{\sqrt{3}}{2}.$
The above inquality of $|\mathcal{S}|$
implies that $|\mathcal{S}| \leq \frac{32\sqrt{3}}{\pi}-9=8.64252$ and thus $|\mathcal{S}|\leq 8$.

Using $2g\leq (|\mathcal{S}|-1)(|\mathcal{S}|-2)/2$, we obtain 
\[g\leq \frac{(|\mathcal{S}|-1)(|\mathcal{S}|-2)}{4}=10.25.\]
Therefore, if $M$ is not Loewner, then $g\leq 10.$
\end{proof}
\begin{remark}
Gromov \cite[Proposition 5.3.A]{gromov1983} proved
 the same inequality 
 $$2g\leq \frac{(N-1)(N-2)}{2},$$
 where the graph $\Gamma$ is given by a short basis of the homology group.  Then he obtain a lower bound  for
 $\frac{\Area(M)}{\sys'(M)^2} $, where $\sys'(M)$ denotes the homological systole. 
\end{remark}

\bibliographystyle{siam}
%\bibliography{myreference.bib}

\end{document}